\documentclass[11pt]{article}

\usepackage{amsmath,amsthm,amssymb,url}
\usepackage{tikz}
\usepackage{tkz-graph}
\usepackage[top=1in, bottom=1in, left=1.25in, right=1.25in]{geometry}

\newcommand{\T}{\ensuremath{\mathcal{T}}}
\newcommand{\A}{\ensuremath{\mathcal{A}}}
\newcommand{\Alg}{\ensuremath{\mathbf{A}}}
\newcommand{\B}{\ensuremath{\mathcal{B}}}
\newcommand{\C}{\ensuremath{\mathcal{C}}}

\newcommand{\lam}{\ensuremath{\lambda}}
\newcommand{\abar}{\ensuremath{\overline{a}}}
\newcommand{\BICA}{\emph{Bulletin of the ICA}}
\newcommand{\DM}{\emph{Discrete Math.}}
\newcommand{\IPL}{\emph{Inform.\ Proc.\ Letters}}
\newcommand{\UM}{\emph{Utilitas Math.}}
\newcommand{\OA}{\ensuremath{\mathrm{OA}}}

\newtheorem{theorem}{Theorem}[section]

\newtheorem{corollary}[theorem]{Corollary}

\theoremstyle{definition}

\newtheorem{remark}[theorem]{Remark}
\title{A Unified Treatment of Some Classic Combinatorial Inequalities Using the Variance Method}
\author{Douglas R.\ Stinson%
\thanks{Research supported by an NSERC General Research Funds grant (University of Waterloo)}
\\David R.\ Cheriton School of Computer Science\\
University of Waterloo\\
Waterloo, Ontario N2L 3G1, Canada
}
\date{\today}

\begin{document}
\maketitle

\begin{abstract}
The ``variance method'' has been used to prove many classical inequalities in design theory and coding theory. The purpose of this expository note is to review and present some of these inequalities in a unified setting. I will also discuss some examples from my own research where I have employed these techniques.
\end{abstract}

\section{The Basic Idea}

The ``variance method'' is used to prove numerous inequalities in design theory and coding theory.
Essentially, we use the fact that the variance of any set of real numbers is non-negative. This underlying theme can be found in the proofs of many famous inequalities, beginning with Fisher's Inequality for BIBDs, which was proven in 1940. Most applications of the variance method begin from first principles, which is not unreasonable. However, it might be useful and instructive to discuss the technique in general. Therefore we derive some general inequalities from which many famous inequalities can be derived by ``plugging in'' suitable expressions depending on the application at hand.

Suppose $a_1, \dots , a_n$ are real numbers. 
For $j = 0,1,2, \dots$, let 
\begin{equation}
\label{Sj.eq}
 S_j = \sum _{i=1}^{n} {a_i}^j.
 \end{equation}
 Thus $S_0 = n$, $S_1$ is the sum of the $a_i$'s, etc.
 
Let $\abar$ denote the mean of the $a_i$'s.
Then
\begin{equation*}
\abar = \frac{S_1}{n} = \frac{S_1}{S_0}.
\end{equation*}
It is straightforward to see that
\begin{align*} 0 &\leq \sum_{i=1}^{n}  (a_i - \abar)^2 \\
&= S_2 - 2\abar S_1 + S_0 \abar^2\\
&= S_2 - 2\frac{{S_1}^2}{S_0} + \frac{{S_1}^2}{S_0}\\
&= S_2 - \frac{{S_1}^2}{S_0}.
\end{align*}
Hence, we have the following fundamental inequality.

\begin{theorem}
\label{ineq1}
 Suppose that $a_1, \dots , a_n$ are real numbers. Define $S_0, S_1$ and $S_2$
as in (\ref{Sj.eq}).
Then
\begin{equation}
\label{eq7}
S_0 S_2 - {S_1}^2 \geq 0.
\end{equation}
\end{theorem}

\begin{remark}
It is interesting to to note that this exact inequality can be found in \cite[p.\ 54]{Fisher}.
Also, it is obvious that equality occurs in (\ref{eq7}) if and only if 
$a_i = \abar$ for $1 \leq i \leq n$.
\end{remark}

\bigskip

In most of the applications we discuss, we  start with an equation for the value
\[S^* = \sum_{i=1}^{n}  \binom{a_i}{2}.\]
It is then a simple matter to compute
\begin{equation}
\label{Sstar} S_2 = 2 S^* + S_1\end{equation}
and then apply the inequality (\ref{eq7}).

\bigskip
 An alternate derivation of inequality (\ref{eq7}) is obtained by observing that it is just a special case of the Cauchy-Schwarz Inequality.  Here are a few details.
 The Cauchy-Schwarz Inequality for real numbers states that
\begin{equation}
\left( \sum _{i=1}^n u_iv_i \right)^2 \leq \left(  \sum _{i=1}^n {u_i}^2 \right) \left(  \sum _{i=1}^n {v_i}^2 \right).
\label{CS.eq}
\end{equation}
Let $u_1 = \cdots = u_n = 1$. Then

\begin{equation}
\left( \sum _{i=1}^n v_i \right)^2 \leq n \left(  \sum _{i=1}^n {v_i}^2 \right).
\label{CS2.eq}
\end{equation}
This is equivalent to (\ref{eq7}).

It is further known that equality occurs in (\ref{CS.eq}) if and only if $(v_1 , \dots , v_n) = c(u_1 , \dots , u_n)$ for some nonzero real number $c$. When $(u_1 , \dots , u_n) = (1, \dots , 1)$, we see that equality occurs in (\ref{CS2.eq}) 
if and only if 
$v_1 = \cdots = v_n$.

\bigskip

In some applications, we will encounter situations where we have a lower bound on $S_1$, rather than an equality.\footnote{We note that a similar situation was addressed in my paper \cite[Theorem 1]{St-J27}, but the proof contained an error.
Amusingly, a stronger result in the same paper had a correct proof. We also note that a correct argument was given in \cite[\S 8.1]{St-B8} in the setting of the Stanton-Kalbfleisch Bound.} 
Suppose that we have equations for $S_0$ and $S^*$, and 
suppose we know that that $S_1 \geq B$, where $B > 0$.  
We can therefore write $S_1 = B + \epsilon$ where  $\epsilon \geq 0$.
We also have $S_2 = 2S^* + B + \epsilon$. Write $C = 2S^* + B$; then $S_2 = C + \epsilon$.
Finally, suppose  that $a_i \geq 1$ for all $i$. Then $S^* \geq 0$ and hence $C = 2S^* + B \geq B > 0$.

From equation (\ref{eq7}), we have
\[ S_0 \geq \frac{(B + \epsilon)^2}{C + \epsilon}.
\] 
Consider $B$ and $C$ to be fixed and define
\[ f(\epsilon) = \frac{(B + \epsilon)^2}{C + \epsilon}.\]
We have 
\begin{align*} 
f'(\epsilon) &= \frac{2(C + \epsilon)(B + \epsilon) - (B + \epsilon)^2}{(C + \epsilon)^2}\\
&= \left( \frac{B + \epsilon}{(C + \epsilon)^2}\right) ( 2(C + \epsilon) - (B + \epsilon)) \\
&= \left( \frac{B + \epsilon}{(C + \epsilon)^2}\right) (2C - B + \epsilon).
\end{align*}
We noted above that  $C \geq B > 0$.
It follows that $2C - B + \epsilon \geq C + C-B > 0$, because $C > 0$ and $C-B \geq 0$. Therefore $f'(\epsilon) > 0$ for all $\epsilon \geq 0$.
Hence, we have that
 \[ S_0 \geq f(\epsilon) \geq f(0) = \frac{B^2}{C}.\]

Therefore we have proven the following.

\begin{theorem}
\label{thm1}
Suppose that $a_1, \dots , a_n$ are real numbers and $a_i \geq 1$ for all $i$. Define $S_0, S_1, S_2$ and $S^*$
as in (\ref{Sj.eq}) and (\ref{Sstar}).
Suppose further that 
$S_1 = B + \epsilon$ and  $C = 2S^* + B$, where $B > 0$ and  $\epsilon \geq 0$.
Then 
\begin{equation}
\label{eq7a}
S_0 C - B^2 \geq 0.\end{equation}
\end{theorem}

In the rest of this note, I derive several classical inequalities for codes and designs  by applying Theorems \ref{ineq1} and \ref{thm1}.  These include Fisher's Inequality for BIBDs (Section \ref{Fisher.sec}), the Plackett-Burman Bound for orthogonal arrays (Section \ref{PB.sec}), the Second Johnson Bound for constant-weight binary codes (Section \ref{John.sec}), the Stanton-Kalbfleisch Bound for pairwise balanced designs (Section \ref{SK.sec}), and the Mullin-Vanstone Bound for $(r, \lambda)$-designs (Section \ref{MV.sec}).  I also discuss a few applications and results from my own research where I have used these techniques, including the two-point sampling derandomization method (Section \ref{twopoint.sec}). Finally, an extension of the variance technique is briefly mentioned in Section \ref{exten.sec}.

\section{Fisher's Inequality}
\label{Fisher.sec}

Fisher's Inequality \cite{Fisher} states that $b \geq v$ (or, equivalently, $r \geq k$, which follows from equation (\ref{BIBD1.eq})) in a $(v,b,r,k,\lam)$-BIBD. Fisher's original proof from 1940 used the variance method. This is actually one of the more complicated ways to prove Fisher's Inequality, as simpler linear algebraic proofs are found much more commonly these days in textbooks (e.g., in \cite{St-B8} as well as many others). However, I believe that this was the first use of the variance method to prove an inequality in combinatorial designs. 

We now present Fisher's original derivation of the eponymous inequality.  A \emph{$(v,b,r,k,\lam)$-BIBD} (balanced incomplete block design) is a collection of $b$ \emph{blocks}, where each block is a subset of $k$ \emph{points} chosen from a $v$-set. 
Every point occurs in exactly $r$ blocks and every pair of points occurs in exactly $\lambda$ blocks.
If $(X, \A)$ is a $(v,b,r,k,\lam)$-BIBD, then the following identities hold:
\begin{align}
\label{BIBD1.eq} vr &= bk\\
\label{BIBD2.eq} \lam(v-1) &= r(k-1).
\end{align}

Denote $\A = \{A_i: 1 \leq i \leq b\}$. For $1 \leq i \leq b-1$, let
$a_i = | A_i \cap A_b|$. 
We have the following equations by simple counting:
\begin{align}
\label{S0.eq}S_0 &= b-1\\
\label{S1.eq}S_1 &= k(r-1)\\
\nonumber S^* &= \binom{k}{2}(\lam -1)\\
\label{S2.eq}S_2 &= k(k-1)(\lam-1) + k(r-1) \quad \text{from (\ref{Sstar})}. 
\end{align}


Now we substitute (\ref{S0.eq}),  (\ref{S1.eq}) and  (\ref{S2.eq}) into
 (\ref{eq7}):
\begin{equation}
\label{eq8}  (b-1)(k(k-1)(\lam -1) + k(r-1)) - k^2(r-1)^2 \geq 0.
\end{equation}

After a considerable amount of non-obvious simplification using the identities (\ref{BIBD1.eq}) and (\ref{BIBD2.eq}), the inequality (\ref{eq8}) can be rewritten as
\begin{equation}\label{eq20} (r-k)(v-k)( r - \lam )   \geq 0.
\end{equation}
For details, see Fisher's original paper \cite{Fisher}. 

\medskip

An ``incomplete block'' design has $k < v$. It then follows from (\ref{BIBD2.eq}) that $\lam < r$. Hence,
\[(v-k)(r - \lam) >0\] and therefore 
(\ref{eq20}) yields the simpler inequality $r - k \geq 0$. This completes the proof of Fisher's Inequality.

\medskip

We see further that, in a BIBD with $v = k$, every pair of blocks intersects in exactly $\abar = S_1 / S_0$ points. However,
if $v=b$, then $\lam (v-1) = k(k-1)$ (from \ref{BIBD2.eq}) and 
\[ \abar = \frac{k(r-1)}{b-1} = \frac{k(k-1)}{v-1} = \lam .\]
So every pair of blocks intersects in exactly $\lam$ points.

\bigskip

Mann's inequality \cite{Mann} (proven in 1969) is another inequality on the parameters of a BIBD. It states that $b \geq sv$ if there is a block that occurs with multiplicity $s$. The proof from \cite{Mann} also uses the variance method; it is very similar to the proof of Fisher's Inequality given in \cite{Fisher}.

\section{Plackett-Burman Bound}
\label{PB.sec}


Let $k \geq 2$, $n \geq 2$ and $\lam \geq 1$ be integers.
An \emph{orthogonal array} $\OA_{\lam} (k,n)$
is a $\lam  n^2 $ by $k$ array, $A$, with entries from a set
$X$ of cardinality $n$ such that, within any two columns of $A$,
every ordered pair of symbols from $X$ occurs in exactly $\lam$ rows.
There are  $\lam  n^2$ rows in the OA and every symbol occurs in exactly $\lam n$ rows within each column of $A$.

We use the variance method to give a simple combinatorial proof of the classical Plackett-Burman bound from 1946. 
(The Plackett-Burman bound is a special case of the better-known Rao bound \cite{Rao}, which was proven in 1947.) In comparison with the proof of Fisher's inequality given in Section \ref{Fisher.sec}, all the remaining proofs in this paper (more precisely, the simplifications of the main inequality) are very simple.

\begin{theorem}[Plackett-Burman Bound (\cite{PB}, p.\ 310)]
\label{PB.thm}
Let $k \geq 2$, $n \geq 2$ and $\lam \geq 1$ be integers. 
If there is an $\OA_{\lam}(k,n)$, then
\[\lam \geq \frac{k(n-1)+1}{n^2}.\]
\end{theorem}

\begin{proof}
Relabel the symbols in each column of the OA so the last row of $A$ is $1 \: 1\: \cdots \: 1$.
For $1 \leq i \leq \lam  n^2-1$, let $a_i$ denote the number of
``$1$''s in row $i$ of $A$.
Then we have the following:
\begin{align*}
S_0 &= \lam n^2-1\\
S_1 &= k(\lam n - 1)\\
S^* &= \binom{k}{2}(\lam - 1)\\
S_2 &= k(k-1)(\lam - 1) + k(\lam n - 1) \quad \text{from (\ref{Sstar})}\\
&= k(k(\lam - 1) + \lam(n-1)).
\end{align*}
Applying the inequality (\ref{eq7}), we obtain 
\[ (\lam n^2-1)(k(k(\lam - 1) + \lam(n-1))) \geq (k(\lam n - 1))^2.\]

Therefore, 
\[ k (\lam n - 1)^2 \leq (\lam n^2 - 1)(k(\lam - 1) + \lam(n-1)).\]
This simplifies to yield
\[k \leq \frac{\lam n^2 - 1}{n-1}.\]
Equivalently, 
\[\lam \geq \frac{k(n-1)+1}{n^2}.\]
\end{proof}

The following result is proven in \cite{CSV,St-J269} using a very similar method. 

\begin{theorem}
\label{repeated.thm}
Let $k \geq 2$, $n \geq 2$ and $\lam \geq 1$ be integers.
If there is an $\OA_{\lam}(k,n)$ containing a row that is repeated $m$ times, then
\[\lam \geq \frac{m(k(n-1)+1)}{n^2}.\]
\end{theorem}

\section{The Second Johnson Bound}
\label{John.sec}

 The Second Johnson Bound (from 1962) for constant-weight binary codes (see \cite[Theorem 3]{Johnson}) is also proven using the variance method. 
As was done in \cite{Johnson}, we first prove a result that applies to certain $0$-$1$ matrices. 

\begin{theorem}
\label{SJ.thm}
Suppose $A = (a_{i,j})$ is an $m$ by $n$ $0$-$1$ matrix such that
\begin{enumerate}
\item every row of $A$ has weight $r$
\item the inner product of any two rows of $A$ is at most $\lambda$, where $r > \lambda$.
\end{enumerate}
Then
\begin{equation}
\label{SJ.eq} m \leq \frac{n(r - \lam)}{r^2 - n \lam},
\end{equation}
provided that ${r^2 - n \lam} > 0$.
\end{theorem}

\begin{proof}
For $1 \leq i \leq n$, define $a_i$ to be the weight of column $i$ of $A$.
It is clear that 
\begin{align}
\nonumber
S_0 &= n\\
\label{eq24}
S_1 &= mr.
\end{align} To compute $S_2$, we count the number of triples 
in the set \[ \T = \{(i,j,k): j < k,  a_{j,i} = a_{k,i} = 1\}.\]
For a fixed  value of $i$, there are $\binom{a_i}{2} $ triples $(i,j,k) \in \T$; hence
\begin{equation}
\label{eq25} | \T | = \sum_{i = 1}^{n} \binom{a_i}{2} .
\end{equation}
On the other hand, if we fix $j$ and $k > j$, there are at most $\lam$ triples $(i,j,k) \in \T$. Hence, 
\begin{equation}
\label{eq26} | \T | \leq \binom{m}{2} \lam.
\end{equation}
Combining (\ref{eq25}) and (\ref{eq26}), we obtain
\begin{equation}
\label{eq27} S^* = \sum_{i = 1}^{n} \binom{a_i}{2} \leq \binom{m}{2} \lam.
\end{equation}
We convert  (\ref{eq27}) to an equality by writing
\begin{equation}
\label{eq28} S^* = \binom{m}{2} \lam - \frac{\epsilon}{2},
\end{equation}
where $\epsilon \geq 0$. Here it will turn out that we can ignore $\epsilon$ without doing any extra work.

\medskip

From (\ref{eq28}) and (\ref{eq24}), 
we have
\begin{equation*}
S_2 =  m(m-1) \lam + mr - \epsilon.
\end{equation*}
As usual, we now apply (\ref{eq7}).
We have
\[ n(m(m-1) \lam + mr - \epsilon) \geq (mr)^2.\]
Because $\epsilon \geq 0$, we get
\begin{align*} n(m(m-1) \lam + mr ) &\geq (mr)^2\\
n((m-1) \lam + r ) &\geq mr^2\\
m(r^2 - n \lam) & \leq n(r - \lam)\\
m &\leq \frac{n(r - \lam)}{r^2 - n \lam} \quad \text{provided that ${r^2 - n \lam} > 0$}.
\end{align*}
This completes the proof.
\end{proof}

An \emph{$[n,r,d]$ constant weight binary code} is a set $\C$ of binary vectors (codewords), each having weight $r$, such that the distance between any two distinct codewords is at least $d > 0$. If the inner product of two distinct codewords is $\lam$, then the distance between the two codewords is $d = 2(r - \lam)$. Since this quantity is even, we typically write $d = 2 \delta$,
 where $\delta = r - \lambda > 0$ is an integer.  We now apply Theorem \ref{SJ.thm}.

\begin{corollary}
\label{John.cor}
 Suppose that $\C$ is an $[n,r,d]$ constant weight binary code with $r > \delta = d/2 > 0$.
Then 
\[ | \C | \leq \frac{n\delta}{r^2 - n (r - \delta)},
\]
provided that $r^2 - n (r - \delta) > 0$.
\end{corollary}

\begin{proof}
Apply the inequality (\ref{SJ.eq}) with $\lambda = r - \delta$. 
Then
\[ | \C | \leq \frac{n(r - (r - \delta))}{r^2 - n (r - \delta)} = \frac{n\delta}{r^2 - n (r - \delta)}.\]
\end{proof}

\begin{remark}
It is interesting to observe that the Plackett-Burman Bound can be derived easily as a corollary of the Second Johnson Bound.
See \cite[\S 3.1]{St-J269}.
\end{remark}

\section{The Stanton-Kalbfleisch Bound}
\label{SK.sec}

In a \emph{pairwise balanced design} (or \emph{PBD}), we have a collection of blocks $\A$ (each having size at least two) in which every pair of points from a given set $X$ occurs in a unique block. Unlike a BIBD, the blocks in a PBD may be of {different sizes}. We now state and prove a well-known bound on the number of blocks in a  PBD that contains at least one block of a specified size $k$. This bound was proven by Stanton and Kalbfleisch in 1972.

\begin{theorem}[Stanton-Kalbfleisch Bound, \cite{SK}]
\label{SKbound.thm}
Let $k$ and $v$ be integers such that 
$2 \leq k < v$. Suppose there is a
PBD on $v$ points in which there exists a 
block containing exactly $k$ points.
Then
\begin{equation}
\label{SK.eq} b \geq 1 + \frac{k^2(v-k)}{v-1}.\end{equation}
\end{theorem}

\begin{proof}
Suppose that $(X,\A)$ is a PBD 
with $|X| = v$ such that $A \in  \A$ is a block containing exactly $k$ points.
Denote the blocks in $\A$ by $A_1, \dots , A_b$, where $A_b = A$.

Now construct a set system $(Y, \B)$ by deleting
all the points in the block $A_b$ as follows:
\begin{align*}
Y &= X \backslash A_b, \\
B_i &= A_i \backslash A_b, 1 \leq i \leq b-1, \quad \mbox{and} \\
\B &= \{ B_i : 1 \leq i \leq b-1 \}.
\end{align*}
$(Y,\B)$ is a set system with $v-k$ points and $b-1$ blocks
in which every pair of points occurs in a unique block.
(This set system may contain blocks of size one, in which case it would technically not be 
a PBD.)

For $1 \leq i \leq b-1$, denote $a_i = |B_i|$.
Note that $a_i = |A_i|$ or $a_i = |A_i| - 1$ for $1 \leq i \leq b-1$.
Furthermore, $a_i = |A_i| - 1$ if and only if $A_i$ intersects
$A_b$ in a point.

Denote the points in $Y$ by $y_j$, $1 \leq j \leq v-k$.
For $1 \leq j \leq v-k$, define 
\[r_j =
| \{ B_i \in \B : y_j \in B_i\} | .\]
Then it is clear that
\begin{equation}
\label{KS.eq0}
\sum_{i=1}^{b-1} a_i = \sum_{j=1}^{v-k} r_j.
\end{equation}

Now, in the pairwise balanced design
$(X,\A)$, every point $y_j$ must occur in a unique block with
each of the points in $A_b$. Hence $r_j \geq k$ for all $j$,
$1 \leq j \leq v-k$. Substituting into (\ref{KS.eq0}), it follows
that 
\begin{equation*}
S_1 = \sum _{i=1}^{b-1} a_i \geq k(v-k).
\end{equation*}
So we have a lower bound on $S_1$, which is the scenario that we discussed in Theorem \ref{thm1}.

\medskip

Every pair of points in $Y$ occurs in exactly one of the $B_i$'s,
so it follows that
\begin{equation*}
S^* = \sum _{i=1}^{b-1} \binom{a_i}{2} = \binom{v-k}{2}.
\end{equation*}
Thus we obtain the following three equations involving $a_1, \dots , a_{b-1}$:
\begin{align*}
S_0 &= b-1\\
S_1 &= k(v-k) + \epsilon\\
S_2 &= (v-k)(v-k-1) + k(v-k) + \epsilon\\
& = (v-k)(v-1) + \epsilon,
\end{align*}
where $\epsilon \geq 0$.

We can apply equation (\ref{eq7a}) from Theorem \ref{thm1}.
It follows that
\[ (b-1) (v-k)(v-1) \geq (k(v-k))^2.\]
This simplifies to yield the inequality (\ref{SK.eq}).
 \end{proof}
 
If equality occurs in the Stanton-Kalbfleisch Bound, then every  block (other than the block of size $k$)  intersects the block of size $k$ in a point, and every such block has exactly 
\[ 1 + \frac{k(v-k)}{b-1} = 1+ \frac{k(v-k)(v-1)}{k^2(v-k)} = 1+ \frac{v-1}{k}\]
points. 

\medskip

\begin{remark}
The Stanton-Kalbfleisch Bound, as presented in \cite{SK}, is more general than the result given in Theorem \ref{SKbound.thm}. The general bound (which we have not stated) actually applies to $\mu$-wise balanced designs for any $\mu \geq 2$. 
Here, we are only discussing the special case $\mu = 2$.
\end{remark}

 
 
 The classic  Erd\H{o}s-de Bruijn Theorem (see \cite{DB-E}) states that $b \geq v$ in a PBD, and $b = v$ if and only if the PBD is a projective plane or a near-pencil. Ralph Stanton observed in his papers (e.g., \cite{SEvRC}) that the  Erd\H{o}s-de Bruijn Theorem is a corollary of the Stanton-Kalbfleisch Bound. 
 The Stanton-Kalbfleisch Bound asserts  that $b \geq 1 + \frac{k^2(v-k)}{v-1}$.
Using this fact, we compute the following:
 \begin{align}
 b - v &\geq 1 + \frac{k^2(v-k)}{v-1} - v \nonumber \\
 & = \frac{k^2(v-k)- (v-1)^2}{v-1}\nonumber \\
 &= \frac{-v^2 + v(k^2 +2) - (k^3+1)}{v-1}\nonumber \\
 &= \frac{-(v - (k + 1))(v - (k^2-k+1))}{v-1}\label{dBE.eq}.
 \end{align}
 Now suppose that $k+1 \leq v \leq k^2-k+1$. 
Then
\[(v - (k + 1))(v - (k^2-k+1)) \leq 0.\]
 From the  inequality (\ref{dBE.eq}), we see that $b \geq v$.

 If $b = v$, then one of the following two  situations occurs:
\begin{enumerate}
\item
$v = k^2-k+1$, every  block  intersects the block of size $k$ in a point and every block has exactly 
$1 + \frac{v-1}{k} = k$ points. This PBD is therefore a projective plane of order $k-1$.
\item $v = k+1$, every  block  intersects the block of size $k$ in a point, and every block (other than the block of size $k$) has size $1 + \frac{v-1}{k} = 2$. This PBD is a near-pencil.
\end{enumerate}

\section{The Mullin-Vanstone Bound and an Application}
\label{MV.sec}

An \emph{$(r,\lambda)$-design} is a PBD of index $\lambda$ (i.e., every pair of points occurs in exactly $\lambda $ blocks) where every point occurs in exactly $r$ blocks.
However, unlike a ``normal'' PBD, blocks of size $1$ are allowed. We next discuss the Mullin-Vanstone Bound \cite{MV}, which was proven in 1975.

\begin{theorem}[Mullin and Vanstone, \cite{MV}]
\label{MVbound.thm}
In any $(r,\lambda)$-design on $v$ points and $b$ blocks, it holds that
 \[ b \geq \frac{r^2v}{r + \lambda(v-1)}.\]
\end{theorem}

The proof follows immediately from  (\ref{eq7}) and the following easily derived equations:
\begin{align*}
S_0 &= b\\
S_1 &= vr\\
S^* &= \lambda \textstyle{\binom{v}{2}}\\
S_2 & = v(r + \lambda(v-1)).
\end{align*}


In 2011, Douglas West asked me for a bound on the largest value of $s$ such that there exists $s$ points and $s$ {nonincident} blocks in a projective plane of order $q$.
This question can be answered using the  following  corollary of the Mullin-Vanstone Bound. 

\begin{corollary} [Stinson \cite{St-J238}]
\label{t1}
If there exists a nonincident set of $s$ points and $t$ lines in a projective
plane of order $q$, then 
\begin{equation}
\label{var-bound}
t \leq \frac{q^3+q^2+q-qs}{q+s}.
\end{equation}
\end{corollary}

\begin{proof}
A projective plane of order $q$ is a $(q+1,1)$-design on $q^2+q+1$ points. Suppose we choose any subset $Y$ of $s$ points in this projective plane. If we intersect each line (block) with $Y$, then we obtain a $(q+1,1)$-design on $s$ points.
Applying Theorem \ref{MVbound.thm}, we see that the number of blocks $b$ in this induced design satisfies the inequality
\[ b \geq \frac{(q+1)^2s}{q+s}.\] Hence, if there are $t$  blocks in the projective plane that are disjoint from $Y$, then
\[ t \leq q^2 + q+1 - \frac{(q+1)^2s}{q+s} = \frac{q^3+q^2+q-qs}{q+s}.\]
\end{proof}

West's question concerned the special case where $t = s$ (this corresponds to an $s$ by $s$ submatrix of 0's in the incidence matrix of the projective plane). 
If we set $s = t$ in (\ref{var-bound}), then it follows that 
\begin{equation}
\label{West.eq}
 s \leq  1 + (q+ 1)(\sqrt{q} - 1).
 \end{equation}
It is  shown in \cite{St-J238} that is possible to attain {equality} in (\ref{West.eq}) if and only if the $s$ points comprise a maximal $(s,\sqrt{q})$-arc in the projective plane. Such maximal arcs exist when 
$q$ is an even power of two; they are known as Denniston arcs.

\section{Two-point Sampling}
\label{twopoint.sec}

Chor and Goldreich \cite{CG} and (independently)  Spencer \cite{Sp} introduced the \emph{two-point sampling} derandomization technique and computed error probabilities of the respective generators.
In 1996, Gopalakrishan and Stinson \cite{GS} observed that these two generators could be unified and generalized by describing them in terms of {orthogonal arrays}, and the error analysis given in \cite{CG,Sp} can be improved using the variance technique.

Here is the setting for two-point sampling. Suppose we use an $\OA_1(k,n)$, say $A$, for derandomization. There are $n^2$ rows in $A$, so a random row of $A$ can be chosen using $2 \log_2 n$ random bits. Once a random row is chosen, we  run a Monte Carlo algorithm, say $\Alg$, $k$ times, using the $k$ elements in the chosen row as randomness for the algorithm $\Alg$. Thus the number of random bits  is reduced from  $k \log_2 n$ (which is what would be required for $k$ random choices) to $2 \log_2 n$. However, the error probability is increased. 

We present the analysis of the error probability of two-point sampling given in \cite{GS}. 
Suppose for the purposes of discussion that the algorithm $\Alg$ is yes-biased with  error probability $\epsilon$. If $I$ is a 
a no-instance, then $\Alg(I)$ always returns ``no.'' However, if $I$ is a 
a yes-instance, then $\Alg(I)$ can erroneously return ``no'' with probability  $\epsilon$. More precisely, for a yes-instance $I$, the number of values $x$ such that  $\Alg(I,x)$ returns ``no'' is  $\epsilon n$. We will call these  \emph{bad points} and denote the set of bad points by $Y$. Points that are not bad are \emph{good points}; denote the set of good points by $Z$. Let $|Z| = z$; then $z = n(1 - \epsilon)$.

For a given instance $I$, we choose a row $r$ of the OA $A$.
Suppose the row $r$ contains the symbols $x_1, \dots , x_k$. Then we run the algorithm  $\Alg(I,x_i)$ for $1 \leq i \leq k$ (so we run the algorithm $k$ times, with randomness specified by the points in row $r$ of $A$). For a yes-instance $I$, we want an upper bound on the probability that $x_1, \dots , x_k$ are all bad points.

For a given row $i$,  define
 $b_i = | \{ j : a_{i,j} \in  Z \} |$ (i.e., $b_i$ is the number of good points in row $i$ of $A$).
Suppose there are $N$ rows with $b_i \geq 1$,  say (without loss of generality) rows $1, \dots , N$.
These are the rows that would allow us to  conclude correctly that $I$ is a yes-instance.
We have the following equations involving the values $b_1, \dots , b_{N}$.
\begin{align*}
S_0 &= N\\
S_1 &= knz\\
S^* &= \binom{k}{2} z^2\\
S_2 &= k(k-1) z^2 + knz = kz((k-1)z+n).
\end{align*} 
Now we apply (\ref{eq7}):
\begin{align*}
 Nkz((k-1)z+n) &\geq (knz)^2\\
N((k-1)z+n) &\geq n^2kz\\
N &\geq \frac{n^2kz}{(k-1)z+n}.
\end{align*}
The number of bad rows is 
\begin{align*}
n^2 - N & \leq n^2 - \frac{n^2kz}{(k-1)z+n}\\
&= n^2 \left( 1 - \frac{kz}{(k-1)z+n} \right).
\end{align*}
Hence the error probability of two-point sampling is at most
\[ 1 - \frac{kz}{(k-1)z+n} .\]
However, $z = n(1 - \epsilon)$. Hence, 
the error probability is at most
\begin{align*}
 1 - \frac{kn(1 - \epsilon)}{(k-1)n(1 - \epsilon)+n} &= 1 - \frac{k(1 - \epsilon)}{(k-1)(1 - \epsilon)+1}\\
 &= \frac{\epsilon}{(k-1)(1 - \epsilon)+1}.
 \end{align*}
 
 \section{An Extension of the Variance Method}
 \label{exten.sec}
 
 So far,  we have discussed the simplest version of the variance method. Extensions are possible, some of which lead to strengthened bounds. Recall that the fundamental idea of the variance method is that a sum of the form 
 \[ \sum_{i=1}^n (a_i - \overline{a})^2 \] non-negative for any set of real numbers $a_1, \dots , a_n$.
 A related approach is based on observation that
\begin{equation}
\label{gen.eq}
 \sum_{i=1}^n (a_i - \ell)(a_i - \ell - 1) \geq 0
 \end{equation} provided that the $a_i$'s are all integers and $\ell$ is an integer. Further, equality holds in (\ref{gen.eq}) if and only if 
 $a_i \in \{ \ell, \ell+1\}$ for all $i$.
 In any given situation, it is usually straightforward to determine the integer $\ell$ that will lead to the strongest bound in (\ref{gen.eq}).
  We briefly discuss a couple of examples to illustrate.
 
 \smallskip
 
 This approach has been used by Stinson in \cite{St-J27} to improve the Stanton-Kalbfleisch Bound.
 
 \begin{theorem}[Stinson Bound,  \cite{St-J27}]
\label{Stbound.thm}
Let $k$ and $v$ be integers such that 
$2 \leq k < v$. Suppose there is a
PBD on $v$ points in which there exists a 
block containing exactly $k$ points. Then for any integer $\ell$, we have
\begin{equation}
\label{St.eq} b \geq 1 + \frac{(2 \ell k - v + k + 1)(v-k)}{\ell^2 + \ell}.\end{equation}
\end{theorem}

\begin{proof}
We use the values of $S_0,S_1,S_2$ and $S^*$ derived in the proof of Theorem \ref{SKbound.thm}, along with (\ref{Sstar}) and (\ref{gen.eq}):
\begin{align*}
0 &\leq \sum_{i=1}^n (a_i - \ell)(a_i - \ell - 1)\\
&= S_2 - (2\ell+1)S_1 + (\ell^2 + \ell)S_0\\
&= 2S^* - 2 \ell S_1 + (\ell^2 + \ell)S_0\\
&\leq (v-k)(v-k-1) - 2 \ell k(v-k) + (\ell^2 + \ell) (b-1)\\
& = (v-k)(v-k-1 - 2 \ell k)+ (\ell^2 + \ell) (b-1)\\
b-1 & \geq \frac{(v-k)(2 \ell k- (v-k-1))}{\ell^2 + \ell}\\
b &\geq 1 + \frac{(2 \ell k - v + k + 1)(v-k)}{\ell^2 + \ell}.
\end{align*}
\end{proof}

The strongest application of Theorem \ref{Stbound.thm} is obtained by taking
\[ \ell = \left\lfloor \frac{v-1}{k} \right\rfloor.\]
For this particular choice of $\ell$, the lower bound given in Theorem \ref{Stbound.thm}  
is always greater than or equal to the Stanton-Kalbfleish Bound; see \cite{St-J27} for details.
 
  \smallskip

 A similar method is used by Johnson in 
\cite{Johnson} in order to provide an improvement of Theorem \ref{SJ.thm}. 
We use the values of $S_0,S_1,S_2$ and $S^*$ derived in the proof of Theorem \ref{SJ.thm}, along with 
(\ref{Sstar}) and (\ref{gen.eq}). The following is obtained:
\begin{align}
0 &\leq 2S^* - 2 \ell S_1 + (\ell^2 + \ell)S_0 \quad \text{as in the proof of Theorem \ref{Stbound.thm}}\nonumber\\
& \leq \lam m (m-1) - 2 \ell mr + (\ell^2 + \ell)n \nonumber\\
m(m-1)\lambda &\geq 2 \ell mr - (\ell^2 + \ell) n.\label{J2.eq}
\end{align}
Following Johnson \cite{Johnson}, we take 
$\ell = \lfloor \frac{mr}{n}\rfloor$ and we write 
\begin{equation}
\label{rem.eq}
mr = n\ell + t,
\end{equation} where $0 \leq t \leq n-1$.
For this value of $\ell$, the inequality (\ref{J2.eq}) can be rewritten as follows:
\begin{align*}
m(m-1)\lambda &\geq 2 \ell (n\ell + t) - (\ell^2 + \ell) n \quad \text{from (\ref{rem.eq})}\\
&= 2 \ell^2 n + 2 \ell t - \ell ^2 n - \ell n\\
&= \ell^2 n +  2 \ell t - \ell n\\
&= \ell^2 n +  2 \ell t - (mr-t) \quad \text{from (\ref{rem.eq})}\\
&= \ell^2 n +  2 \ell t + t - mr\\
&= \ell^2 n - \ell^2 t + \ell^2 t +  2 \ell t  + t - mr\\
&= (n-t)\ell^2  + t (\ell+1)^2 - mr.
\end{align*} We have proven the following.
 
 \begin{theorem}[\cite{Johnson}, eq.\ (6)]
Suppose $A = (a_{i,j})$ is an $m$ by $n$ $0$-$1$ matrix such that
\begin{enumerate}
\item every row of $A$ has weight $r$
\item the inner product of any two rows of $A$ is at most $\lambda$, where $r > \lambda$.
\end{enumerate}
Let $\ell = \lfloor \frac{mr}{n}\rfloor$ and write  $mr = n\ell + t$,
where $0 \leq t \leq n-1$.
Then
\begin{equation}
 \label{john.eq} 
 m(m-1)\lambda \geq (n-t)\ell^2  + t (\ell+1)^2 - mr.
\end{equation}
\end{theorem}
If we regard $n$, $r$ and $\lambda$ as being fixed, we would then compute the largest integer $m$ that satisfies
(\ref{john.eq}). Of course, this is just a necessary condition for the existence of the desired $0$-$1$ matrix.



\begin{thebibliography}{XX}

\bibitem{DB-E}
N.G. de Bruijn and P. Erd\H{o}s. On a combinatorial problem.
\emph{Indagat. Math.} {\bf 10} (1948), 421--423.

\bibitem{CG}
B. Chor and O. Goldreich. On the power of two-point based
sampling. \emph{J. Complexity} {\bf 5} (1989), 96--106.

\bibitem{CSV}
C.J.\ Colbourn, D.R.\ Stinson and S.~Veitch.
Constructions of optimal orthogonal arrays with repeated rows.
\DM \ \textbf{342} (2019), 2455--2466.


\bibitem{Fisher}
R.A. Fisher. An examination of the different possible solutions of a problem in incomplete blocks.
\emph{Annals of Eugenics} {\bf 10} (1940), 52--75.

\bibitem{GS}
K.~Gopalakrishnan and D.R. Stinson.
A simple analysis of the error probability of two-point based
sampling.  
\IPL \ \textbf{60} (1996), 91--96.

\bibitem{Johnson} S.M. Johnson. A new upper bound for error-correcting codes, \emph{IRE
Trans. on Inform. Theory} {\bf 8} (1962), 203--207.

\bibitem{Mann}
H.B. Mann. A note on balanced incomplete-block designs, \emph{Ann. Math.
Statist.} {\bf 40} (1969), 679--680.

\bibitem{MV} R.C.\ Mullin and S.A.\ Vanstone.
On regular pairwise balanced designs of order 6 and index 1.
\emph{Utilitas Math.} {\bf 8} (1975), 349--369.


\bibitem{PB} R.L. Plackett and J.P. Burman. The design of optimum multifactorial
experiments. \emph{Biometrika} {\bf 33} (1946), 305--325.

\bibitem{Rao}
C.R. Rao. Factorial experiments derivable from combinatorial arrangements of arrays. \emph{Suppl. J. Roy. Statist.
Soc.} {\bf 9} (1947), 128--139.

\bibitem{Sp}
T.H. Spencer. Provably good pattern generators for a random
pattern test. \emph{Algorithmica} {\bf 11} (1994), 429--442.

\bibitem{SEvRC}
R.G. Stanton, P. Eades, J. van Rees and D.D. Cowan. (1980), 
Computation of some exact $g$-coverings.
\UM \ {\bf 18} (1980), 269--282.

\bibitem{SK}
R.G. Stanton and J.G. Kalbfleisch. The $\lambda$-$\mu$ problem: $\lambda = 1$ and $\mu = 3$,
In ``Proc. Second Chapel Hill Conf. on Combinatorics'' (1972), pp. 451--462.

\bibitem{St-J27}
D.R. Stinson.
Applications and generalizations of the variance method in
combinatorial designs.
\UM \ \textbf{22} (1982), 323--333.


\bibitem{St-B8} 
D.R. Stinson. 
{\em Combinatorial Designs: Constructions and Analysis}. Springer-Verlag, 
New York, 2004.

\bibitem{St-J238}
D.R. Stinson.
Nonincident points and blocks in designs.
\DM \  \textbf{313} (2013), 447--452.

\bibitem{St-J269} 
D.R. Stinson. Bounds for orthogonal arrays with repeated rows.
 \BICA \ \textbf{85} (2019), 60--73.



\end{thebibliography}
\end{document}